\documentclass[reqno, 11pt]{amsart}
\usepackage{amssymb,amsmath,amsfonts,layout}
\usepackage[margin=1in]{geometry}
\usepackage{dsfont}
\usepackage{color}
\usepackage{latexsym}
\usepackage{graphicx}
\usepackage{graphics}
\usepackage[dvips]{epsfig}
\usepackage{mathrsfs}
\usepackage{mathtools}
\usepackage{leqno}
\usepackage{stmaryrd,cite}

\usepackage{cases}
\usepackage{enumerate}
\usepackage[usenames,dvipsnames,svgnames]{xcolor}
\definecolor{refkey}{rgb}{1,0,0.5}
\definecolor{labelkey}{rgb}{0,0.4,1}
\usepackage{todonotes}
\usepackage{marginnote}

\presetkeys{todonotes}{fancyline, color=green!40}{}

\usepackage{listings}
\usepackage{ifpdf}
\ifpdf
\usepackage[CJKbookmarks=true,
hyperindex=true,
pdfstartview=FitH,
bookmarksnumbered=true,
bookmarksopen=true,
colorlinks=true,
citecolor=blue,
linkcolor=blue,
urlcolor=blue,
pdfborder=001,
pdfauthor={},
pdftitle={},
pdfkeywords={},
]{hyperref}
\else
\usepackage[%CJKbookmarks=true,
hypertex,
hyperindex,
linkcolor=blue,
unicode,
citecolor=blue
]{hyperref}
\fi
\usepackage{cleveref}

\allowdisplaybreaks
\numberwithin{equation}{section}
\newtheorem{theorem}{Theorem}[section]

\newtheorem{lemma}[theorem]{Lemma}

\newtheorem{remark}[theorem]{Remark}

\newcommand{\da}{\delta}

\newcommand{\ga}{\gamma}
\newcommand{\ba}{\beta}

\newcommand{\Da}{\Delta}

\newcommand{\Ga}{\Gamma}
\newcommand{\Rn}{\mathbb R^n}
\newcommand{\RN} {\mathbb{R}^N}
\newcommand{\Rm}{\mathbb R^m}

\newcommand{\R}{\mathbb R}
\newcommand{\N}{\mathbb N}
\newcommand{\be}{\begin{equation}}
	\newcommand{\ee}{\end{equation}}

\newcommand{\bee}{\begin{equation*}}
	\newcommand{\eee}{\end{equation*}}

\newcommand{\bse}{\begin{subequations}}
	\newcommand{\ese}{\end{subequations}}

\newcommand{\bs}{\begin{split}}
	\newcommand{\es}{\end{split}}

\begin{document}
	
	\author[Qiu]{Yusheng Qiu$^{1}$}\thanks{$^{1}$School of Mathematics and Statistics, Jiangxi Normal University, Nanchang, Jiangxi 330022, China.
		E-mail: yusheng.qiu@jxnu.edu.cn}
	
	\author[Tan]{Jinggang Tan$^{2}$}\thanks{$^{2}$School of Mathematics and Statistics, Jiangxi Normal University, Nanchang, Jiangxi 330022, China.
		E-mail : jinggang.tan@usm.cl}
	
	\author[Xia]{Aliang Xia$^{3}$}\thanks{$^{3}$School of Mathematics and Statistics\, \&\, Jiangxi Provincial Center for Applied Mathematics, Jiangxi Normal University, Nanchang, Jiangxi 330022, China.
		E-mail: aliang\_xia@jxnu.edu.cn}

	\title[QCP to fourth-order Baouendi-Grushin type ] {Quantitative unique continuation property for fourth-order Baouendi-Grushin type subelliptic operators with a potential}

	\maketitle
	
	\begin{abstract}
		We investigate the quantitative unique continuation property for solutions to
		\begin{equation*}
			\Da^2_{X} u = V u,
		\end{equation*}
		where $\Da_{X} = \Da_{x} + |x|^{2\ba} \Da_{y}$ ($0 < \ba \leq 1$), with $x \in \mathbb{R}^{m}$ and $y \in \mathbb{R}^{n}$, denotes a class of subelliptic operators of Baouendi-Grushin type. The potential $V$ is assumed to be bounded and satisfy $|Z V| \leq K \psi$ for some constant $K>0$,
		where $Z= \sum_{i=1}^m x_i \partial_{x_i} + (\ba+1)\sum_{j=1}^n y_j \partial_{y_j}$, $\psi$ is the angle function given by $\psi = \frac{|x|^{2\ba}}{\rho^{2\ba}}$,
		 and $$\rho(x,y) = \left(|x|^{2(\beta+1)} + (\beta+1)^2 |y|^2\right)^{\frac{1}{2(\beta+1)}}$$ defines the associated pseudo-gauge. By adapting Almgren's approach, we establish an almost monotonicity formula for the frequency function. As a consequence, we derive a quantitative unique continuation result for solutions to the fourth-order subelliptic equation.
		
		
	\end{abstract}
	
	\section{Introduction}\label{intro}
	
	This note is concerned with the quantitative unique continuation property of solutions to the fourth-order Baouendi-Grushin equation perturbed by zero-order perturbation. 
	Baouendi-Grushin operator  that arises in  sub-Riemannian geometry, and control theory  \cite{Gr1}, \cite{Gr2}, \cite{B}  is degenerate elliptic operator as follows 
	\begin{align}\label{Grushin0}
		\Da_{X}=\Da_{x}+|x|^{2\ba}\Da_{y}, \quad 0<\ba\leq1,
	\end{align}
	which is the operator of a sum of squares of
	the following vector fields $X=\{X_1,\cdots,X_{m+n}\}$,
	\begin{equation}\label{vector1}
		X_i=\partial_ {x_i},
		\hspace{1mm}i=1,\cdots,m,\hspace{2mm}
		X_{m+j}=|x|^{\ba} \partial_{y_j},\hspace{1mm}j=1,\cdots,n, \quad m,n\geq1.
	\end{equation}
	Here $\beta \in (0,1]$ is a fixed parameter, $x=(x_1,\cdots,x_{m})\in
	\mathbb{R}^{m}$, $y=(y_1,\cdots,y_{n})\in \mathbb{R}^{n}$ and $\Delta_{x}, \Delta_{y}$ denote the standard Laplacian. 
	
	\medskip
	
	We observe that the operator is degenerate elliptic on the characteristic manifolds $\left\lbrace 0\right\rbrace \times \Rn$, and it is elliptic when $x\neq0$. For $\ba=0$, the operator $\Da_{X}$ is the classical Laplacian operator. We mention that if $\ba=1$, the operator $\Da_{X}$ is connected to the sub-Laplacians on the Heisenberg-type groups (see for example \cite{GR}). 
	The operator  $\Da_{X}$,   $\ba\in\N$,  \cite{Gr1,Gr2}   is hypoelliptic operator related to
	singular Cauchy-Riemann equations  which statisfies H\"ormander-type condition, \cite{H}.  Note that the Kalman rank condition may fail, but hypoellipticity could still ensure controllability.  
	
	\medskip
	
	We are interested in explicit dependence of solutions vanishing on an open subset or satisfying certain asymptotic decay conditions.  In general cases, a function $u\in L^{2}$ is said to have vanishing  order $k\ge0$ at some point $x_{0}\in\mathbb{R}^{N}$ if 
	\[
	\frac{\int_{B_{r}(x_{0})}u^{2} dx}{r^{N+2k}}=O(1), \quad \text{as}~r \rightarrow 0,
	\]
	and vanish to infinite order at some point $x_{0}$ if
	\[
	\int_{B_{r}(x_{0})}u^{2} dx=o(r^{N+2k})\quad\mbox{for any integer}\,\,k.
	\]	
	In a connected domain $\Omega$, a differential operator $L$ is said to have the strong unique continuation property  if the solution to the equation $Lu= 0$ that vanishes to infinite order at a point $x_{0}\in\Omega$ then  $u = 0$ in $\Omega$. If $L$ satisfies the strong unique continuation property, nontrivial solutions to $Lu = 0$ do not vanish in infinite order. The quantitative unique continuation property is   to characterize the vanishing orders of solutions by the coefficient functions appeared in the equations $Lu = 0$.
	
	\medskip
	
	The quantitative unique continuation property and strong unique continuation property  for solutions to various kinds of  partial differential equations have attracted a lot of attention in recent decades, and the main methods to solve these problems are Carleman estimate and frequency function argument. Precisely,
	for the stationary Schr\"odinger equation
	\begin{align}\label{vv}
		-\Delta u = V(x)u,
	\end{align}
	Carleman presented in \cite{Ca} that when the potential function $V(x)$ is bounded, the solutions of the equation in $\R^2$ have a unique continuation property. Subsequent developments in this direction is Cordes  \cite{Co} and Aronszajn  \cite{Ar} proved the strong unique continuation theorem of the problem in the case of $\RN$. In \cite{JK}, Jerison and Kenig derived the strong unique continuation property
	of the differential inequality $|\Delta u| \leq |V(x)u|$ when $V \in L_{loc}^{\frac{N}{2}} ( \RN)$. 
	On the other hand,  Garofalo and Lin \cite{GL1,GL2} presented  a geometric-variational approach to the strong unique continuation property by using the frequency function, which was first introduced by Almgren \cite{Al} for harmonic function.
	
	\medskip
	
	To quantitative unique continuation property, it is well known that the vanishing order of the solutions to  the eigenvalue problems on compact smooth Riemannian manifolds
	\begin{align*}
		-\Delta_g \varphi_\lambda = \lambda \varphi_\lambda,
	\end{align*}
	is  less than $C\sqrt{\lambda}$, where $C$ depends only on the manifold, see for example \cite{DF1,DF2,L}. Kukavica \cite{Ku} has shown  that the vanishing order of the solutions to equation \eqref{vv} is less than $C\Big(1+\|V^{-}\|_{L^\infty}^{\frac{1}{2}} + \mbox{osc}V + \|\nabla V\|_{L^\infty}\Big)$ if $V \in W^{1,\infty}$, and Bakri \cite{Ba} and Zhu \cite{Zhu} obtained  that the vanishing order is less than $C\big( 1 + \|V\|_{W^{1,\infty}}^{\frac{1}{2}}\big) $. Bourgain and Kenig  \cite{BK} found out that the vanishing order is not large than $C\big( \|V\|_{L^\infty}^{\frac{2}{3}}\big) $. We also refer to \cite{Da, ZD, KT} and references therein for more quantitative unique continuation property of solutions to second order elliptic equations.
	
	\medskip
	
	Consider  the Grushin-type Schr\"odinger equation of the form $$\Da_{X}u=Vu.$$
	The strong unique continuation property was confirmed in Garofalo   \cite{Ga}, by means of the frequency function argument for the Grushin-type equation. For the quantitative result, Banerjee, Garofalo and Manna \cite{BGM} showed that if $|V| \leq K\psi$ then solutions in $B_R$ satisfy
	\[
	\|u\|_{L^\infty(B_r)}\geq C_1\left(\frac{r}{R_0}\right)^{C_2(K^{\frac{2}{3}}+1)}
	\]
	for $0<r<\frac{R_0}{9}$ with $R_0 \in (0, \frac{R}{2}]$.
	Moreover, the operator $$\Da_{X}-b\cdot\nabla u-V$$ 
	were also studied by  Garofalo \cite{Ga}
	under some appropriate assumptions. Garofalo and Vassilev \cite{GV} testified  the strong unique continuation property 
	the following type equations
	\begin{align*}
		\sum_{i,j=1}^N X_i (a_{ij}(x,y) X_j u)=0,
	\end{align*}
	under suitable assumptions on $a_{ij}$, see also \cite{AKS}. The new  frequency function which added a weight term $(r^2-|x|^2)^\alpha$ was developed by Kukavica\cite{Ku} and Zhu\cite{Zhu}. Garofalo and Banerjee in \cite{BG1} applied this new frequency function to obtain a sharper result.\par
	\medskip
	Consider the equation
	\begin{align*}
		\sum_{i,j=1}^N X_i (a_{ij}(x,y) X_j u)=V(x,y)u,
	\end{align*} 
	Garofalo and Banerjee employed weight and frequency function in \cite{BG2}, getting the quantitative unique continuation property  in the following 
	\begin{align*}
		\|u\|_{L^{\infty}(B_r)} \geq C_1 \left(\frac{r}{R_1}\right)^{C_2 \sqrt{K}},
	\end{align*}  where $R_1$ and $K$ are positive constants in the suitable hypotheses. Whereas without the differentiability assumption on $V$, Garofalo and Manna \cite{BM} demonstrate that
	\[
	\|u\|_{L^\infty(B_r)}\geq C_1\left(\frac{r}{R_0}\right)^{C_2(K^{\frac{2}{3}}+1)},
	\]
	where $R_0$ depends on the domain of solutions.
	
	\medskip
	
	Zhu \cite{Zhu} verified that the value of $\alpha$ in the weight function $(r^2-|x|^2)^\alpha$ can be used to calculate the coefficient in Hadamard's three-ball theorem, which provides better vanishing order and eliminates the boundary terms while using the divergence theorem. It also proved \cite{Zhu} that the vanishing order of $u$ is less than $\|V\|_{L^{\infty}}$ for $n\ge 4m$ 
	to  the higher order equations with a potential
	\begin{align*}
		(-\Delta)^mu=Vu\quad\mathrm{in} \ \ \ \Rn,
	\end{align*}
	by a variant of  frequency function.
	In \cite{LTY}, the authors demonstrated the upper bound of maximal vanishing order of solutions to the bi-Laplacian (that is $m=2$).
	
	\medskip
	
	Liu and Yang \cite{LY} considered the fourth-order Grushin-type equation with a bounded potential as
	\begin{align*}
		\Da^2_{X}u=Vu\quad\mathrm{in} \ \ \ \Omega,
	\end{align*}
	where $|V|\leq \frac{c_0}{\rho^4}$, and proved the strong unique continuation property using the frequency function without weight. It is mentioned that, in subelliptic  cases, a function $u\in L^{2}_{loc}$ is said to vanish to infinite order at origin if 
	\[
	\int_{B_{r}}u^{2}\psi dxdy=O(r^{k}), \quad\mbox{for every }\,k\in\mathbb{N},
	\]
	as $r\rightarrow0$, with $B_r$ denoting the pseudo-ball, and $\psi$ is given by \eqref{psi} below.  Moreover, we say that $u$ vanishes at the origin more rapidly than any power of $k$, that is 
	\[
	\int_{B_{r}}u^{2}\psi dxdy=O\left(e^{-cr^{-\gamma}}\right),
	\]
	as $r\rightarrow0$ for some constants $c,\gamma>0$.
	
	\medskip
	
	We here focus on the above  fourth-order Grushin-type equation with a   potential and  characterize the vanishing orders of solutions. Precisely,  the main result of this paper is the following theorem.
	
	\begin{theorem}\label{thm}
		Let $m>2$, $0<\ba\leq1$, and $B_r$ is the gauge pseudo-ball centered at $0$ with radius $r$ (see \eqref{ball} below). If $u$ is a solution of
		\begin{equation} \label{equ1}
			\Da^2_{X}u=V(x,y)u \quad \mbox{in} \quad B_{1},
		\end{equation}
		where $V$ satisfies 
		\begin{equation}\label{vasump}
			|V| \leq K_1, \quad|Z V| \leq K_2 \psi,
		\end{equation}	
		for some $K_1, K_2\ge0$, and $|u|\leq C_0$, where $Z$ is a vector field defined by \eqref{Z}. Furthermore, assume $X_iX_ju \in L_{loc}^2(B_1)$ for $i,j=1,2,\cdots,m+n$. Then there exist positive numbers $c=c(m, n, \beta, K_1, C_0, u)$, $C=C(m, n, \beta,K_2, C_0, u)$ and $\overline{C}=\overline{C}(m)$ such that
		\begin{equation}\label{main}
			\left\|u\right\|_{L^\infty(B_{r})} \geq cr^{C\overline{C}^{K_1} K_1^2},
		\end{equation}
		for any $0<r<1$.
	\end{theorem}
	
	\begin{remark}
		\eqref{main} implies that the vanishing order of $u$ is bounded above by $\overline{C}^{K_1} K_1^2$, with $\overline{C}$ depending only on the dimension $m$ of $x$.
	\end{remark}
	
	\begin{remark}
		In fact, \eqref{main} holds with the constants $c$, $C$ and $\overline{C}$ behaved like
		\begin{equation*}
			\begin{aligned}
				\log c &= \log \left(\frac{C_0}{C(Q)K_1 \int_{B_{1/4}}\psi}\right) -4(\log 2) K_1^2 e^{8\left( \max\left\{1, \left(\frac{2}{m-2}\right)^2\right\} +2 \right) K_1}\\[2mm]
				&\quad\times\max\left\{4(m-2) , \frac{16}{(m-2)^2} , \frac{\log\frac{C_0}{h(\frac{1}{3})}}{\log\left(\frac{3}{2}\right)}\right\},
		\end{aligned}\end{equation*}
		\begin{equation*}
			\begin{aligned}
				C = 2\max\left\{8(m-2), \frac{32}{(m-2)^2} + C(K_2), \frac{4\log\frac{C_0}{h\left(\frac{1}{3}\right)}}{\log\left(\frac{3}{2}\right)} , 4\right\},
			\end{aligned}
		\end{equation*}
		\begin{equation*}
			\overline{C} = e^{8\max\left\{1, \left(\frac{2}{m-2}\right)^2\right\} +2 },
		\end{equation*}
		where $Q = m + (\ba+1) n$, $\psi$ is defined by \eqref{psi}, and the value of $h(\frac{1}{3})$ can be obtained from \eqref{h} only depending on $\beta$ and $u$. In particular, when both $m$ and $n$ are sufficient large, Theorem \ref{thm} infers that
		\begin{equation*}
			\left\|u\right\|_{L^\infty(B_{r})} \geq cr^{8(m-2)e^{10K_1}K_1^2}.
		\end{equation*}
	\end{remark}
	
	One of  ingradients is to construct a generalized Almgren-type frequency function introduced by \cite{Zhu}. The classical Almgren frequency function is defined as the ratio between the quadratic form associated with the equation and the $L^2$-norm of the trace of solutions, integrated over a ball of radius $r$ and its boundary, respectively. In the new frequency function proposed here, the boundary integral is replaced by integration over the interior of the domain combining with a weight function $(r^2-\rho^2)^\alpha$, which vanishes on the boundary. To give the estimate of the derivative of the frequency function, we establish its almost monotonicity
	formula of such frequency function
	by means of  the estimate of the derivatives of the associated energy functional.
	
	\medskip
	
	To overcome this, we develop a Hardy-type inequality and a Rellich-type identity to control the problematic terms. It is worth noting that Lemma \ref{Hardy} is a Hardy-type inequality for second-order equations rather than fourth-order equations. Therefore, we apply this lemma twice to complete the proof. Due to the difficuty of the Baouendi-Grushin operator, the definition of the vanishing order involves the complexity of $\psi$,  the role of the doubling estimate employed in \cite{LTY} has replaced by Hadamard's three-ball theorem which  come from the monotonicity property of frequency function. Finally, adaptting  additional regularity results, including a Caccioppoli-type inequality and a Moser iteration result, we deduce the  the maximum norm of associated height function.
	
	\medskip 	The rest of the paper is organized as follows. In Section \ref{nota}, we introduce some basic notations and preliminary results associated to the Grushin-type operators. In Section \ref{mono}, we decompose the equation \eqref{equ1} into the appropriate system of second-order equations. We also provide the almost monotonicity of the frequency function. In Section \ref{three}, we obtain Hadamard's three-ball theorem for the Grushin-type subelliptic system, which is applied to prove our main result in Section \ref{maxi}.

	\section{Notations and preliminary results}\label{nota}
	Let $\{X_i\}$ for $i=1, \cdots, N$ be defined as in \eqref{vector1}. Denote an arbitrary point in $\RN$ as  $(x,y) \in \Rm \times \Rn$. Given a function $f$, we denote
	\begin{equation*}
		Xf= (X_1f, \cdots,X_Nf),\ \ \ \ \ \ \ \ \ |Xf|^2= \sum_{i=1}^N (X_i f)^2,
	\end{equation*}
	respectively the Grushin gradient and the square of its length. Now given a vector function $F=\left(f_1, \cdots,f_N\right)$, we also define the Grushin divergence associated to  the vector fields \eqref{vector1} as follows
	\begin{equation}\label{divx}
		{\rm div}_X F = \sum_{i=1}^N X_i f_i.
	\end{equation}
	We recall from \cite{Ga} that the following family of anisotropic dilations are associated with the vector fields in \eqref{vector1},
	\begin{equation}\label{dil}
		\da_a(x,y)=(a x,a^{\ba+1} y),\ \ \ \ \ \ \ \ a>0.
	\end{equation}
	which generate the homogeneous dimension for $\Da_{X}$,
	\begin{equation*}
		Q= m + (\ba+1) n,
	\end{equation*}
	since denoting by $dxdy$ Lebesgue measure in $\R^N$, we have $d(\da_a(x,y)) = a^Q dx dy$.
	In addition, one can find the following remarkable fundamental solution $\Ga$ of $\Da_{X}$  with pole at the origin is given by the formula
	\[
	\Gamma(x,y) = \frac{C}{\rho(x,y)^{Q-2}},\ \ \ \ \ \ \ \ \ (x,y)\not= (0,0),
	\]
	where $\rho$ is the pseudo-gauge 
	\begin{equation}\label{rho}
		\rho(x,y)=(|x|^{2(\beta+1)} + (\beta+1)^2 |y|^2)^{\frac{1}{2(\beta+1)}}.
	\end{equation}
	We respectively denote by 
	\begin{equation}\label{ball}
		B_r = \{(x,y)\in \R^N\mid \rho(x,y) < r\},\ \ \ \ \ \ \ \ S_r = \{(x,y)\in \R^N\mid \rho(x,y) = r\},
	\end{equation}
	the gauge pseudo-ball and sphere centered at $0$ with radius $r$. 
	
	We also need the angle function $\psi$ introduced in \cite{Ga},
	\begin{equation}\label{psi}
		\psi = |X\rho|^2= \frac{|x|^{2\ba}}{\rho^{2\ba}}.
	\end{equation}
	The function $\psi$ vanishes on the characteristic manifold $M=\{0\} \times \Rn$ and clearly satisfies $0\leq \psi \leq 1$. 
	
	The generator of the group of dilations \eqref{dil}  is given by the following smooth vector field
	\begin{equation}\label{Z}
		Z= \sum_{i=1}^m x_i \partial_{x_i} + (\ba+1)\sum_{j=1}^n y_j \partial_{y_j}.
	\end{equation}
	A  function $v$ is $\da_a$-homogeneous of degree $\kappa$ if and only if $Zv=\kappa v$. Since $\rho$ in \eqref{rho} is homogeneous of degree one, we have
	\begin{equation}\label{hg}
		Z\rho=\rho.
	\end{equation}
	
	We note the important facts that
	\begin{equation}\label{divZ}
		\operatorname{div} Z = Q,\quad \quad [X_i,Z] = X_i,\ \ \ i=1,\cdots,N,
	\end{equation}
	where $[X_i,Z] = X_i Z-Z X_i$ denotes their commutator.
	Directly computation implies that
	\begin{equation}\label{Xrho}
		X\rho=\frac{1}{\rho^{2\beta+1}}\left(|x|^{2\beta}x,(\beta+1)|x|^{\beta}y\right),
	\end{equation}
	hence
	\begin{equation}\label{Zeq}
		Z= \frac{\rho}{\psi} \sum_{k=1}^N X_k\rho X_k,
	\end{equation}
	and
	\begin{equation}\label{Zueq}
		Zu= \frac{\rho}{\psi} \langle X\rho, Xu \rangle,
	\end{equation}
	with $\langle \cdot ,\cdot  \rangle$ denoting the standard inner product.
	Since $\psi$ is homogeneous of degree zero with respect to \eqref{dil}, one has
	\begin{equation}\label{Zpsi}
		Z\psi = 0.
	\end{equation}
	
	We recall that  $V$ in \eqref{equ1} satisfies the following hypothesis for some $K_1,K_2\ge 0$,
	\begin{equation}
		|V| \leq K_1, \quad |ZV| \leq K_2 \psi,
	\end{equation}
	where $\psi$ indicates the function introduced in \eqref{psi} above. Without loss of generality we assume henceforth that $K_1 = \|V\|_{L^\infty(B_1)} \geq 1$ and $K_2 = \|ZV\|_{L^\infty(B_1)} \geq 1$.

	\section{Monotonicity of a generalized frequency}\label{mono}
	
	In this section,we decompose \eqref{equ1} into the appropriate system of second-order equations. It is the major goal to establish the almost monotonicity of frequency function $N(r)$ in this section. Finally with the help of the monotonicity property, we can obtain Hadamard's three-ball theorem of $h(r)$ which is the height function without weight.
	
	\medskip
	
	Let $z=(x,y)$ with $z\in\mathbb{R}^{m+n}$. We decompose equation (\ref{equ1}) into the following system of two second-order equations:
	\begin{equation}\label{equ}
		\left\{\begin{array}{lll}
			\Delta_X u =  w,\\[2mm]
			\Delta_X w=V(z)u.
		\end{array}\right.
	\end{equation}
	
	For simplicity, we denote  $\|V\|_{L^{\infty}}\equiv\|V\|_{L^{\infty}(B_1)}$ and $\|ZV\|_{L^{\infty}}\equiv\|ZV\|_{L^{\infty}(B_1)}$. Throughout this paper, we omit the Lebesgue measure $dxdy$ in the integrals.
	For $0< r <1$, we define the generalized height function $H(r)$ of $u$ in $B_r$ as follows
	
	\begin{equation}\label{H}
		H(r)= \int_{B_r} (u^2+w^2) (r^2-\rho^2)^{\alpha} \psi, 
	\end{equation}
	where $\rho$ is the pseudo-gauge in \eqref{rho} above, and $\alpha > 0$ is going to be fixed later.
	We also define
	\begin{align}\label{I}
		I(r) =   \int_{B_r}|Xu|^2(r^2-\rho^2)^{\alpha+1}+\int_{B_r}|Xw|^2(r^2-\rho^2)^{\alpha+1} +\int_{B_r}\left(1+V\right)uw(r^2-\rho^2)^{\alpha+1},
	\end{align}
	and
	\begin{equation}\label{N}
		N(r) = \frac{I(r)}{H(r)}
	\end{equation}
	as the generalized energy functional and the generalized frequency of $u$ in $B_r$.
	We can also write $I$ in another form, as shown in the following lemma.
	
	\begin{lemma}\label{Iequal}
		Let $u$ and $w$ satisfy \eqref{equ}, then
		\begin{equation}\label{Iequ}
			I(r)=  2(\alpha +1)\int_{B_r}  (uZu+wZw) (r^2-\rho^2)^{\alpha}\psi,
		\end{equation}
		where $I(r)$ is defined in \eqref{I}.
	\end{lemma}
	\begin{proof}
		By (\ref{Zueq}) and the system \eqref{equ}, applying the divergence theorem, we can get
		\begin{align*}
			I(r) &= \int_{B_r}|Xu|^2(r^2-\rho^2)^{\alpha+1}+\int_{B_r}|Xw|^2(r^2-\rho^2)^{\alpha+1}\nonumber\\[2mm]
			&\quad+\int_{B_r} u\Delta_X u(r^2-\rho^2)^{\alpha+1}+\int_{B_r}w\Delta_X w(r^2-\rho^2)^{\alpha+1}\nonumber\\[2mm]
			&=\int_{B_r}\mbox{div}_X\left(uXu+wXw\right)(r^2-\rho^2)^{\alpha+1}\nonumber\\[2mm]
			&=-\int_{B_r}\langle\left(uXu+wXw\right),X(r^2-\rho^2)^{\alpha+1}\rangle\nonumber\\[2mm]
			&=2(\alpha+1)\int_{B_r}\langle\left(uXu+wXw\right),\rho (X\rho) (r^2-\rho^2)^\alpha\rangle\nonumber\\[2mm]
			&=2(\alpha +1)\int_{B_r}  (uZu+wZw) (r^2-\rho^2)^{\alpha}\psi.
		\end{align*}
	\end{proof}
	
	The main result of this section is the following almost monotonicity for the generalized frequency $N(r)$.
	\begin{theorem}\label{T:mono}
		Let u and w satisfy the equations \eqref{equ}, then the function
		\[
		r \mapsto e^{C\|V\|_{L^\infty}r^2}\big(N(r) + C'\|V\|_{L^\infty}^2(\alpha+1) + \frac{1}{4C} \frac{\|ZV\|_{L^\infty}}{\|V\|_{L^\infty}} \big)
		\]
		is monotone non-decreasing on the interval $(0,1)$, where $C$ and $C'$ are positive numbers only depending on $m$.
		
		Moreover,
		\begin{align*}
			C = 4\max\left\{1, \left(\frac{2}{m-2}\right)^2\right\} + 1,
		\end{align*}
		\begin{align*}
			C' = 2(m-2) \max \left\{1, \frac{4}{(m-2)^3}\right\}.
		\end{align*}
	\end{theorem}
	The proof of Theorem \ref{T:mono} will be divided into several steps.
	Now we start to prove the first variation formula for $H(r)$.
	
	\begin{lemma}\label{H1}
		For every $0<r<1$, one has
		\begin{equation}\label{hr}
			H'(r)= \frac{2\alpha + Q }{r} H(r) + \frac{1}{(\alpha+1)r} I(r),
		\end{equation}
		where $H(r)$ is given by (\ref{H}).
	\end{lemma}
	
	%{\textbf{Proof}.}
	\begin{proof}
		By differentiate (\ref{H}), noting that $(r^2-\rho^2)^{\alpha}$ vanishes on $S_{r}$, we get
		\begin{align}\label{H'}
			H'(r) & = 2\alpha r \int_{B_r}  (u^2+w^2) (r^2 - \rho^2)^{\alpha-1} \psi\nonumber
			\\[2mm]
			& = \frac{2\alpha} {r} \int_{B_r}  \left(u^2+w^2\right)\left(r^2 - \rho^2\right)^{\alpha} \psi\nonumber
			\\[2mm]
			& \quad + \frac{2\alpha} {r} \int_{B_r}  (u^2+w^2) \rho^2 (r^2 - \rho^2)^{\alpha - 1} \psi\nonumber
			\\[2mm]
			& =: \frac{2\alpha} {r} H(r) + K(r),
		\end{align}
		which uses the identity
		\[
		(r^2 - \rho^2)^{\alpha-1} = \frac{1}{r^2} (r^2 - \rho^2)^{\alpha} + \frac{\rho^2}{r^2}(r^2 - \rho^2)^{\alpha-1}.
		\]
		By (\ref{hg}), we can get
		\begin{equation}\label{Zr}
			Z(r^2-\rho^2)^{\alpha} = -2\alpha\rho^2(r^2-\rho^2)^{\alpha-1}.
		\end{equation}
		Using (\ref{Zpsi}), (\ref{divZ}), (\ref{Zr}), and integration by parts, one has
		\begin{align}\label{K}
			K(r)&=-\frac{1}{r}\int_{B_r}\left(u^2+w^2\right)Z(r^2-\rho^2)^{\alpha}\psi\nonumber\\[2mm]
			&=\frac{1}{r}\int_{B_r}\mbox{div}\left((u^2+w^2)Z\psi\right)(r^2-\rho^2)^{\alpha}\nonumber\\[2mm]
			&=\frac{Q}{r}\int_{B_r}(u^2+w^2)(r^2-\rho^2)^{\alpha}\psi\nonumber\\[2mm]
			&\quad +\frac{2}{r}\int_{B_r}(uZu+wZw)(r^2-\rho^2)^{\alpha}\psi.
		\end{align}
		Substituting (\ref{K}) into (\ref{H'}), and recalling \eqref{Iequ} in Lemma \ref{Iequal}, the proof is completed.
	\end{proof}
	
	Next we estimate $I'(r)$. Before doing this, we are going to prove a local version of Hardy-Rellich type inequality as below. In \cite{LY}, the authors has proved some refined Hardy inequalities without weight. 
	
	\begin{lemma}\label{Hardy}
		Let $m>2$. If $u\in L^2(B_r)$ and $Xu\in L^2(B_r)$, then it holds
		\begin{align}\label{hardy1}
			\int_{B_r}\frac{u^2}{|x|^2}(r^2-\rho^2)^{\alpha+1}\leq \left(\frac{2}{m-2}\right)^2\int_{B_r}|Xu|^2(r^2-\rho^2)^{\alpha+1}+\frac{4(\alpha+1)}{m-2}\int_{B_r}u^2(r^2-\rho^2)^{\alpha}\psi.
		\end{align}
		Moreover, if $0<\beta\leq 1$, then
		\begin{align}\label{hardy}
			\int_{B_r}\frac{u^2}{\rho^2\psi}(r^2-\rho^2)^{\alpha+1}\leq \left(\frac{2}{m-2}\right)^2\int_{B_r}|Xu|^2(r^2-\rho^2)^{\alpha+1}+\frac{4(\alpha+1)}{m-2}\int_{B_r}u^2(r^2-\rho^2)^{\alpha}\psi.\\ \nonumber
		\end{align}
	\end{lemma}
	\begin{proof}
		Let
		\begin{equation*}
			h=\frac{1}{|x|^2}
			\left(
			\begin{array}{lll}
				x \\
				0
			\end{array}\right)
			\in \mathbb{R}^{m+n},
		\end{equation*}
		with $x\in \mathbb{R}^{m}$.
		Recalling \eqref{Xrho}, direct calculation shows that
		\begin{equation}\label{h-1}
			\mbox{div}_{X}h=\frac{m-2}{|x|^2}, \hspace{3mm} \quad
			\langle h,X\rho\rangle=\rho^{-(2\beta+1)}|x|^{2\beta}=\rho^{-1}\psi.
		\end{equation}
		Hence, using (\ref{divx}), (\ref{h-1}), integrating by parts and H\"older's inequality, we have
		\begin{align*}
			(m-2)\int_{B_r}\frac{u^2}{|x|^2}(r^2-\rho^2)^{\alpha+1}&=\int_{B_r}\mbox{div}_{X}h u^2(r^2-\rho^2)^{\alpha+1} \\[2mm]
			&=-\int_{B_r}\langle h,X\left[u^2(r^2-\rho^2)^{\alpha+1}\right]\rangle\\[2mm]
			&=-2\int_{B_r}\langle h,Xu\rangle  u(r^2-\rho^2)^{\alpha+1}+2(\alpha+1)\int_{B_r}u^2(r^2-\rho^2)^{\alpha}\rho\langle h,X\rho\rangle\\[2mm]
			&\leq2\int_{B_r}\frac{|u||Xu|}{|x|}(r^2-\rho^2)^{\alpha+1}+2(\alpha+1)\int_{B_r}u^2(r^2-\rho^2)^{\alpha}\psi\\[2mm]
			&\leq\frac{m-2}{2}\int_{B_r}\frac{ u^2}{|x|^2}(r^2-\rho^2)^{\alpha+1}+\frac{2}{m-2}\int_{B_r}|Xu|^2(r^2-\rho^2)^{\alpha+1}\\[2mm]
			&\quad +2(\alpha+1)\int_{B_r}u^2(r^2-\rho^2)^{\alpha}\psi,
		\end{align*}
		which can yield (\ref{hardy1}). Furthermore, if $\beta\leq 1$, then $\rho^2\psi\geq|x|^2$, so we obtain (\ref{hardy}).
	\end{proof}
	
	\begin{lemma}\label{I'r}
		Let u and w satisfy the equations \eqref{equ}, then for every $0<r<1$, one has
		\begin{align}\label{ir}
			I'(r)&\geq \frac{2\alpha + Q }{r} I(r) + \frac{4(\alpha+1)}{r} \int_{B_r} (Zu)^2(r^2-\rho^2)^{\alpha} \psi + \frac{4(\alpha+1)}{r} \int_{B_r} (Zw)^2(r^2-\rho^2)^{\alpha} \psi \nonumber\\
			&\quad - C_1 \|V\|_{L^\infty}^3 (\alpha+1)rH(r)- \frac{1}{2} \|ZV\|_{L^\infty}rH(r) - C_2 \|V\|_{L^\infty} r I(r),
		\end{align}
		where
		\begin{align*}
			C_1 = 32 (m-2)\max\left\{1, \left(\frac{2}{m-2}\right)^2\right\} \max\left\{1, \frac{4}{(m-2)^3}\right\},
		\end{align*}
		\begin{align*}
			C_2 = 8\max\left\{1, \left(\frac{2}{m-2}\right)^2\right\} + 2.
		\end{align*}
	\end{lemma}
	
	\begin{proof}
		By recalling the definition of $I(r)$ in \eqref{I}, we denote $I_i(r)(i = 1, 2, 3)$ in the form of
		\begin{equation}\label{I-0}
			\begin{aligned}
				I(r)&=\int_{B_r}|Xu|^2(r^2-\rho^2)^{\alpha+1}+\int_{B_r}|Xw|^2(r^2-\rho^2)^{\alpha+1}\\[2mm]
				&\quad+\int_{B_r}\left(1+V\right)uw(r^2-\rho^2)^{\alpha+1}\\[2mm]
				&=:I_1(r)+I_2(r)+I_3(r).
			\end{aligned}
		\end{equation}
		Now we estimate $I_i'(r)$($i=1,2,3$) one by one.
		\begin{align}\label{1}
			I_1'(r)&=2(\alpha+1)r\int_{B_r}|Xu|^2(r^2-\rho^2)^{\alpha}\nonumber\\[2mm]
			&=\frac{2(\alpha+1)}{r}\int_{B_r}|Xu|^2(r^2-\rho^2)^{\alpha+1}\nonumber\\[2mm]
			&\quad+\frac{2(\alpha+1)}{r}\int_{B_r}|Xu|^2\rho^2(r^2-\rho^2)^{\alpha}.
		\end{align}
		To deal with the second term in the r.h.s. of (\ref{1}), we need the following lemma, which can also be deduced by using the Rellich-type identity in Lemma 2.11 in \cite{GV}:
		\begin{lemma}\label{iden}
			If $Xu\in L^2(B_r)$ and $\Delta_X u\in L^2(B_r)$, then it holds
			\begin{align}\label{con1}
				& \int_{B_r} |Xu|^2 \rho^2(r^2-\rho^2)^{\alpha} = \frac{Q-2}{2(\alpha+1)}\int_{B_r} |Xu|^2 (r^2- \rho^2)^{\alpha+1}\notag\\[2mm]
				&+ 2\int_{B_r}(Zu)^2(r^2- \rho^2)^{\alpha}\psi - \frac{1}{\alpha+1}\int_{B_r}Zu\Delta_X u(r^2-\rho^2)^{\alpha+1}.
			\end{align}
		\end{lemma}
		
		\begin{proof}
			By \eqref{hg}, we can get the following identity,
			\begin{equation}\label{fact}
				Z(r^2-\rho^2)^{\alpha+1} = -2(\alpha+1)\rho^2 (r^2-\rho^2)^\alpha,
			\end{equation}
			which holds that
			\begin{equation}\label{351}
				\int_{B_r} |Xu|^2 \rho^2 (r^2-\rho^2)^\alpha = -\frac{1}{2(\alpha+1)}\int_{B_r}|Xu|^2 Z(r^2-\rho^2)^{\alpha+1}.
			\end{equation}
			Using integration by parts, we obtain
			\begin{align}\label{352}
				\int_{B_r}|Xu|^2 Z(r^2-\rho^2)^{\alpha+1} &= -Q\int_{B_r}|Xu|^2 (r^2-\rho^2)^{\alpha+1}\notag\\
				&\quad-2\sum_{i=1}^{m+n}\int_{B_r}(X_i u) (ZX_i u) (r^2-\rho^2)^{\alpha+1}.
			\end{align}
			Thanks to \eqref{divZ} and the integration by parts, we have
			\begin{align}\label{353}
				&\sum_{i=1}^{m+n}\int_{B_r}(X_i u) (ZX_i u) (r^2-\rho^2)^{\alpha+1}\notag\\[2mm]
				&= \sum_{i=1}^{m+n}\int_{B_r}(X_i u) (X_i Zu) (r^2-\rho^2)^{\alpha+1} - \int_{B_r}|Xu|^2 (r^2-\rho^2)^{\alpha+1}\notag\\[2mm]
				&= -\int_{B_r}Zu\Delta_X u(r^2-\rho^2)^{\alpha+1} + 2(\alpha+1)\int_{B_r}Zu\langle X\rho , Xu\rangle \rho(r^2-\rho^2)^\alpha\notag\\[2mm]
				&\quad- \int_{B_r}|Xu|^2 (r^2-\rho^2)^{\alpha+1}.
			\end{align}
			By the fact \eqref{Zueq}, we can calculate
			\begin{equation}\label{354}
				\int_{B_r}Zu\langle X\rho , Xu\rangle\rho(r^2-\rho^2)^\alpha = \int_{B_r} (Zu)^2 (r^2-\rho^2)^\alpha\psi.
			\end{equation}
			Hence, substituting \eqref{352}, \eqref{353}, \eqref{354} into \eqref{351}, we finally come by \eqref{con1}.  
		\end{proof}
		
		Now come back to the proof of Lemma \ref{I'r}.
		
		Recalling the definition of $I_1(r)$ in (\ref{I-0}) and the first equation in (\ref{equ}), combining Lemma \ref{iden} with (\ref{1}), we obtain
		\begin{align}\label{I'1}
			I'_1(r)=&\frac{2\alpha+Q}{r}I_1(r)+\frac{4(\alpha+1)}{r}\int_{B_r}(Zu)^2(r^2-\rho^2)^{\alpha}\psi\nonumber\\[2mm]
			&\quad-\frac{2}{r}\int_{B_r} wZu(r^2-\rho^2)^{\alpha+1}\nonumber\\[2mm]
			& =: \frac{2\alpha+Q}{r}I_1(r)+\frac{4(\alpha+1)}{r}\int_{B_r}(Zu)^2(r^2-\rho^2)^{\alpha}\psi\nonumber\\[2mm]
			&\quad+ R_1.
		\end{align}
		Similarly, using Lemma \ref{iden}, we can calculate $I'_2(r)$ as
		\begin{align}\label{I'21}
			I_2'(r)&=2(\alpha+1)r\int_{B_r}|Xw|^2(r^2-\rho^2)^{\alpha}\nonumber\\[2mm]
			&=\frac{2(\alpha+1)}{r}\int_{B_r}|Xw|^2(r^2-\rho^2)^{\alpha+1}+\frac{2(\alpha+1)}{r}\int_{B_r}|Xw|^2\rho^2(r^2-\rho^2)^{\alpha}\nonumber\\[2mm]
			&=\frac{2(\alpha+1)}{r}\int_{B_r}|Xw|^2(r^2-\rho^2)^{\alpha+1}+\frac{Q-2}{r}\int_{B_r} |Xw|^2 (r^2- \rho^2)^{\alpha+1}\notag\\[2mm]
			&\quad+ \frac{4(\alpha+1)}{r}\int_{B_r}(Zw)^2(r^2- \rho^2)^{\alpha}\psi - \frac{2}{r}\int_{B_r}Zw\Delta_X w(r^2-\rho^2)^{\alpha+1},
		\end{align}
		which implies that
		\begin{align*}
			I_2'(r)&=\frac{2\alpha+Q}{r}I_2(r)+\frac{4(\alpha+1)}{r}\int_{B_r}(Zw)^2(r^2-\rho^2)^{\alpha}\psi\nonumber\\[2mm]
			&\quad -\frac{2}{r}\int_{B_r} VuZw(r^2-\rho^2)^{\alpha+1}\nonumber\\[2mm]
			& =: \frac{2\alpha+Q}{r}I_2(r)+\frac{4(\alpha+1)}{r}\int_{B_r}(Zw)^2(r^2-\rho^2)^{\alpha}\psi\nonumber\\[2mm]
			&\quad + R_2.
		\end{align*}
		
		For
		\begin{align*}
			I_3(r) = \int_{B_r}\left(1+V\right)uw(r^2-\rho^2)^{\alpha+1},
		\end{align*}
		by integrating by parts, we obtain
		\begin{align*}
			I'_3(r) &= 2(\alpha+1)r\int_{B_r}(1+V)uw(r^2-\rho^2)^{\alpha}\notag\\[2mm]
			& = \frac{2(\alpha+1)}{r}\int_{B_r}(1+V)uw(r^2-\rho^2)^{\alpha+1}\notag\\[2mm]
			&\quad + \frac{2(\alpha+1)}{r}\int_{B_r}(1+V)uw\rho^2(r^2-\rho^2)^{\alpha}\notag\\[2mm]
			& = \frac{2(\alpha+1)}{r}\int_{B_r}(1+V)uw(r^2-\rho^2)^{\alpha+1}\notag\\[2mm]
			&\quad + \frac{1}{r}\int_{B_r}\mbox{div}\left(\left(1+V\right)uwZ\right)(r^2-\rho^2)^{\alpha+1}\notag\\[2mm]
			& = \frac{2(\alpha+1)}{r}I_3(r) + \frac{Q}{r}I_3(r)\notag\\[2mm]
			&\quad + \frac{1}{r}\int_{B_r}uZw(r^2-\rho^2)^{\alpha+1} + \frac{1}{r}\int_{B_r}VuZw(r^2-\rho^2)^{\alpha+1}\notag\\[2mm]
			& \quad+ \frac{1}{r}\int_{B_r}wZu(r^2-\rho^2)^{\alpha+1} + \frac{1}{r}\int_{B_r}VwZu(r^2-\rho^2)^{\alpha+1}\notag\\[2mm]
			&\quad + \frac{1}{r}\int_{B_r}uwZV(r^2-\rho^2)^{\alpha+1}\notag\\[2mm]
			& =: \frac{2\alpha+Q}{r}I_3(r) + \frac{2}{r}I_3(r)\notag\\[2mm]
			&\quad + R_3^1+R_3^2+R_3^3+R_3^4+R_3^5.
		\end{align*}
		Now summing the $I'_i(r)$ ($i=1,2,3$), we finally get that
		\begin{align*}
			I'(r) &= \frac{2\alpha+Q}{r}I(r) + \frac{4(\alpha+1)}{r}\int_{B_r}(Zu)^2(r^2-\rho^2)^{\alpha}\psi\\[2mm]
			&\quad + \frac{4(\alpha+1)}{r}\int_{B_r}(Zw)^2(r^2-\rho^2)^{\alpha}\psi + \frac{2}{r}I_3(r)\\[2mm]
			&\quad + R_1+R_2\\[2mm]
			&\quad + R_3^1+R_3^2+R_3^3+R_3^4+R_3^5.
		\end{align*}
		Next we start to estimate other bad terms such as $\frac{2}{r}I_3(r)$, $R_1$, $R_2$, $R_3^1$, $\cdots$, $R_3^5$. In order to estimate the terms with $Zu$ or $Zw$, we show the following facts
		\begin{equation}\label{Zu}
			\begin{aligned}
				|Zu|  = \frac{\rho}{\psi}|\langle Xu,X\rho \rangle| = \rho \psi^{-\frac{1}{2}}\frac{|\langle Xu,X\rho \rangle|}{|X\rho|} 
				\leq  \rho \psi^{-\frac{1}{2}}|Xu|,
			\end{aligned}
		\end{equation}
		by using (\ref{Zueq}).
		For
		\[
		R_1 = -\frac{2}{r}\int_{B_r} wZu(r^2-\rho^2)^{\alpha+1},
		\]
		by recalling the Hardy inequality in (\ref{hardy}) in Lemma \ref{Hardy}, (\ref{I'1}) and (\ref{Zu}),
		\begin{align*}
			\left|R_1\right|&\leq\frac{2}{r}\int_{B_r}|wZu|(r^2-\rho^2)^{\alpha+1}\notag\\[2mm]
			&\leq2\int_{B_r}\frac{\rho}{r}\psi^{-\frac{1}{2}}|Xu||w|(r^2-\rho^2)^{\alpha+1}\notag\\[2mm]
			&\leq2r\int_{B_r}\frac{|Xu||w|}{\rho\psi^{\frac{1}{2}}}(r^2-\rho^2)^{\alpha+1}\notag\\[2mm]
			&\leq r\int_{B_r}|Xu|^2(r^2-\rho^2)^{\alpha+1} + r\int_{B_r}\frac{w^2}{\rho^2\psi}(r^2-\rho^2)^{\alpha+1}\notag\\[2mm]
			&\leq \max\left\{1, \left(\frac{2}{m-2}\right)^2\right\} r\left(\int_{B_r}|Xu|^2(r^2-\rho^2)^{\alpha+1} + \int_{B_r}|Xw|^2(r^2-\rho^2)^{\alpha+1}\right)\notag\\[2mm]
			&\quad + \frac{4}{m-2} (\alpha+1)r\int_{B_r}w^2(r^2-\rho^2)^{\alpha}\psi\notag\\[2mm]
			&\leq \max\left\{1, \left(\frac{2}{m-2}\right)^2\right\}r\left(I_1(r)+I_2(r)\right) + \frac{4}{m-2}(\alpha+1)r H(r).
		\end{align*}
		Similarly, for $R_2$, applying (\ref{vasump}), we obtain
		\begin{align*}
			\left|R_2\right|&\leq\frac{2}{r}\|V\|_{L^\infty}\int_{B_r} uZw(r^2-\rho^2)^{\alpha+1}\notag\\[2mm]
			&\leq\frac{2}{r}\|V\|_{L^\infty}\int_{B_r} uZw(r^2-\rho^2)^{\alpha+1}\notag\\[2mm]
			&\leq \|V\|_{L^\infty}r\int_{B_r}|Xw|^2(r^2-\rho^2)^{\alpha+1} + \|V\|_{L^\infty}r\int_{B_r}\frac{u^2}{\rho^2\psi}(r^2-\rho^2)^{\alpha+1}\notag\\[2mm]
			&\leq \max\left\{1, \left(\frac{2}{m-2}\right)^2\right\}\|V\|_{L^\infty}r\left(\int_{B_r}|Xw|^2(r^2-\rho^2)^{\alpha+1} + \int_{B_r}|Xu|^2(r^2-\rho^2)^{\alpha+1}\right)\notag\\[2mm]
			&\quad+ \frac{4}{m-2}\|V\|_{L^\infty}(\alpha+1)r\int_{B_r}u^2(r^2-\rho^2)^{\alpha}\psi\notag\\[2mm]
			&\leq \max\left\{1, \left(\frac{2}{m-2}\right)^2\right\}\|V\|_{L^\infty}r\left(I_1(r)+I_2(r)\right)+\frac{4}{m-2}\|V\|_{L^\infty}(\alpha+1)r H(r).
		\end{align*}
		For $R_3^1$, $R_3^2$, $R_3^3$ and $R_3^4$, using the similar method again, with \eqref{vasump} in hand, we have
		\begin{align*}
			\left|R_3^1+R_3^3\right|\leq \max\left\{1, \left(\frac{2}{m-2}\right)^2\right\}r\left(I_1(r)+I_2(r)\right) + \frac{4}{m-2}(\alpha+1)r H(r),
		\end{align*}
		\begin{align*}
			\left|R_3^2+R_3^4\right|\leq \max\left\{1, \left(\frac{2}{m-2}\right)^2\right\}\|V\|_{L^\infty}r\left(I_1(r)+I_2(r)\right) + \frac{4}{m-2}\|V\|_{L^\infty}(\alpha+1)r H(r).
		\end{align*}
		And for $R_3^5$,
		by the definition of $H(r)$ in (\ref{H}) and Young's inequality, one has
		\begin{align*}
			\left|R_3^5\right|&\leq\frac{\|ZV\|_{L^\infty}}{2r}\int_{B_r}(u^2+w^2)(r^2-\rho^2)^{\alpha+1}\psi\notag\\[2mm]
			&\leq\frac{\|ZV\|_{L^\infty}}{2}r\int_{B_r}(u^2+w^2)(r^2-\rho^2)^{\alpha}\psi\notag\\[2mm]
			&=\frac{1}{2}\|ZV\|_{L^\infty}rH(r).
		\end{align*}
		By the definition of $I_3(r)$ in \eqref{I},
		using (\ref{hardy}) and Young's inequality, we can calculate
		\begin{align*}
			\left|\frac{2}{r}I_3(r)\right|&\leq\frac{2}{r}\int_{B_r}\left(1+\|V\|_{L^\infty}\right)|uw|(r^2-\rho^2)^{\alpha+1}\notag\\[2mm]
			&\leq\frac{2}{r}\left(1+\|V\|_{L^\infty}\right)\left(\epsilon\int_{B_r}\frac{u^2}{\psi}(r^2-\rho^2)^{\alpha+1} + \frac{1}{4\epsilon}\int_{B_r}w^2(r^2-\rho^2)^{\alpha+1}\psi\right)\notag\\[2mm]
			&\leq2\left(1+\|V\|_{L^\infty}\right)r\left(\epsilon\int_{B_r}\frac{u^2}{\rho^2\psi}(r^2-\rho^2)^{\alpha+1} + \frac{1}{4\epsilon}\int_{B_r}w^2(r^2-\rho^2)^{\alpha}\psi\right)\notag\\[2mm]
			&\leq2\left(1+\|V\|_{L^\infty}\right)r\left(\epsilon\left(\frac{2}{m-2}\right)^2\int_{B_r}|Xu|^2(r^2-\rho^2)^{\alpha+1}\right.\notag\\[2mm]
			&\quad +\epsilon\frac{4(\alpha+1)}{m-2}\int_{B_r}u^2(r^2-\rho^2)^{\alpha}\psi + \left.\frac{1}{4\epsilon}\int_{B_r}w^2(r^2-\rho^2)^{\alpha}\psi\right)\notag\\[2mm]
			&\leq 2\left(\frac{2}{m-2}\right)^2 \epsilon \left(1+\|V\|_{L^\infty}\right) r I_1(r) + 2 \max\left\{\frac{4\epsilon}{m-2}, \frac{1}{4\epsilon}\right\} \left(1+\|V\|_{L^\infty}\right) (\alpha+1)r H(r),
		\end{align*}
		which implies
		\begin{equation*}
			\left|I_3(r)\right|\leq \left(\frac{2}{m-2}\right)^2 \epsilon \left(1+\|V\|_{L^\infty}\right)r^2 I_1(r) + \max\left\{\frac{4\epsilon}{m-2}, \frac{1}{4\epsilon}\right\} \left(1+\|V\|_{L^\infty}\right) (\alpha+1)r^2 H(r).
		\end{equation*}
		Thus we have
		\begin{align*}
			&I_1+I_2 = I-I_3 \\[2mm]
			&\leq I + \left(\frac{2}{m-2}\right)^2 \epsilon \left(1+\|V\|_{L^\infty}\right) r^2 (I_1 + I_2) +  \max\left\{\frac{4\epsilon}{m-2}, \frac{1}{4\epsilon}\right\} \left(1+\|V\|_{L^\infty}\right) (\alpha+1)r^2 H.
		\end{align*}
		Choosing $\epsilon = \frac{(m-2)^2}{8\left(1+\|V\|_{L^\infty}\right)}$, it yields
		\begin{equation}\label{I1I2}
			I_1+I_2 \leq 2I + (m-2) \max\left\{1, \frac{4}{(m-2)^3}\right\} \left(1+\|V\|_{L^\infty}\right)^2(\alpha+1)r^2 H(r).
		\end{equation}
		%	since $$C(\epsilon) = \frac{1}{4\epsilon} = \frac{1}{2}C(1+K)(1+\alpha).$$
		Finally by (\ref{I1I2}) we get 
		\begin{align*}
			&\quad \left|R_1 + R_2 + R_3^1 +R_3^2 + R_3^3 +R_3^4\right|\\[2mm]
			&\leq 2(m-2)M_1 M_2 \left(1+\|V\|_{L^\infty}\right)^3 (\alpha +1)rH(r) + 4 M_1 \left(1+\|V\|_{L^\infty}\right)r I(r),
		\end{align*}
		and
		\begin{align*}
			\left|R_3^5\right| \leq \frac{1}{2}\|ZV\|_{L^\infty}rH(r),
		\end{align*}
		where $M_1 := \max\left\{1, \left(\frac{2}{m-2}\right)^2\right\}$, $M_2 := \max\left\{1, \frac{4}{(m-2)^3}\right\}$.
		%		$$\left|R_5^5\right| \leq \frac{1}{2} KrH(r)$$
		For $\left|\frac{2}{r}I_3(r)\right|$, recalling (\ref{vasump}),
		\begin{align*}
			\left|\frac{2}{r}I_3(r)\right| \leq 2(m-2)M_2 \left(1+\|V\|_{L^\infty}\right)^2 (\alpha+1)rH(r) + 2rI(r).
		\end{align*}
		Hence we conclude that
		% In fact,
		%		\begin{align*}
			%			I'(r)&\geq \frac{2\alpha + Q }{r} I(r) + \frac{4(\alpha+1)}{r} \int_{B_r} (Zu)^2(r^2-\rho^2)^{\alpha} + \frac{4(\alpha+1)}{r} \int_{B_r} (Zw)^2(r^2-\rho^2)^{\alpha}\nonumber\\
			%			&- C_1 \left(\left(1+K\right)^3 + K \right) (\alpha+1)rH(r)-C_2 \left(1+K\right)r I(r),
			%		\end{align*}
		%		\begin{align*}
			%			I'(r)&\geq \frac{2\alpha + Q }{r} I(r) + \frac{4(\alpha+1)}{r} \int_{B_r} (Zu)^2(r^2-\rho^2)^{\alpha} + \frac{4(\alpha+1)}{r} \int_{B_r} (Zw)^2(r^2-\rho^2)^{\alpha}\nonumber\\
			%			&- C_1 K^3 (\alpha+1)rH(r)-C_2 Kr I(r),
			%		\end{align*}
		%		where $C_1 := \frac{81}{2} (m-1) M_1 M_2$, $C_2:= 8M_1+4$ only depending on $m$.
		\begin{align*}
			I'(r)&\geq \frac{2\alpha + Q }{r} I(r) + \frac{4(\alpha+1)}{r} \int_{B_r} (Zu)^2(r^2-\rho^2)^{\alpha} \psi + \frac{4(\alpha+1)}{r} \int_{B_r} (Zw)^2(r^2-\rho^2)^{\alpha} \psi\\
			&\quad - C_1 \|V\|_{L^\infty}^3 (\alpha+1)rH(r) - \frac{1}{2}\|ZV\|_{L^\infty}rH(r) -C_2 \|V\|_{L^\infty}r I(r),
		\end{align*}
		where $C_1 : = 32(m-2)M_1M_2$, $C_2:= 8M_1+2$ only depending on $m$.
	\end{proof}
	
	Now thanks to Lemma \ref{I'r}, we can provide the almost monotonicity of frequency function $N(r)$.
	
	\begin{proof}[Proof of Theorem \ref{T:mono}]
		Recalling the definition of $N(r)$ in (\ref{N}), we can estimate
		\begin{align}\label{N'}
			N'(r) &= \frac{I'(r)H(r)-I(r)H'(r)}{H^2(r)}\notag\\[2mm]
			& \geq \frac{1}{H^2(r)}\left\lbrace\left[\frac{4(\alpha+1)}{r}\int_{B_r}(Zu)^2(r^2-\rho^2)^\alpha \psi + \frac{4(\alpha+1)}{r}\int_{B_r}(Zw)^2(r^2-\rho^2)^\alpha \psi \right]\cdot H(r) \right.\notag\\[2mm]
			& \quad \left.- \frac{1}{(\alpha+1)r}I^2(r) - C_1 \|V\|_{L^\infty}^3 (\alpha+1)rH^2(r) - \frac{1}{2}\|ZV\|_{L^\infty}rH^2(r) - C_2 \|V\|_{L^\infty}r I(r)H(r)\right\rbrace\notag\\[2mm]
			& \geq \frac{1}{H^2(r)}\left\lbrace - C_1 \|V\|_{L^\infty}^3 (\alpha+1)rH^2(r) -\frac{1}{2}\|ZV\|_{L^\infty}rH^2(r) -  C_2 \|V\|_{L^\infty}r I(r)H(r)\right\rbrace\notag\\[2mm]
			& = - C_1 \|V\|_{L^\infty}^3 (\alpha+1)r - \frac{1}{2}\|ZV\|_{L^\infty}r - C_2 \|V\|_{L^\infty}r N(r),
		\end{align}
		where in the last inequality, we have used the Cauchy-Schwarz inequality
		\begin{align*}
			\frac{4(\alpha+1)}{r}\left[\int_{B_r}(Zu)^2(r^2-\rho^2)^\alpha \psi + \int_{B_r}(Zw)^2(r^2-\rho^2)^\alpha \psi \right]\cdot H(r) - \frac{1}{(\alpha+1)r}I^2(r) \geq 0,
		\end{align*}
		in view of the definition of $H(r)$ in (\ref{H}) and Lemma \ref{Iequal}.
		Thus the inequality (\ref{N'}) implies that
		%			\[
		%			r \mapsto e^{C_3 K r^2} \big(N(r) + C_4 K^2 (\alpha+1)\big)
		%			\]
		\[
		r \mapsto e^{C_3\|V\|_{L^\infty}r^2}\left(N(r)+C_4\|V\|_{L^\infty}^2(\alpha+1)+\frac{1}{4C_3}\frac{\|ZV\|_{L^\infty}}{\|V\|_{L^\infty}}\right)
		\]
		is nondecreasing, where $C_3 = C_2/2$, $C_4 = C_1/C_2$.
	\end{proof}
	
	\section{Hadamard's three-ball theorem}\label{three}
	
	In this section, we will give a $L^2$-version of Hadamard's three-ball theorem. In order to get rid of the weight function $(r^2-\rho^2)^\alpha$, we introduce the height function without weight. Let
	\begin{equation}\label{h}
		h(r) = \int_{B_r}(u^2+w^2)\psi,
	\end{equation}
	then it is easy to verify that
	\begin{equation}\label{hge}
		H(r) \leq r^{2\alpha} h(r),
	\end{equation}
	and
	\begin{equation}\label{hle}
		h(r) \leq \frac{H(s)}{(s^2-r^2)^\alpha}
	\end{equation}
	for any $0<r<s<1$. We can get the following three-ball theorem.
	
	\begin{theorem}\label{tb}
		Let $0<r_1<r_2<2r_2<r_3<1$, then there exist positive numbers $C_5 = C_5(m)$ and $C_6 = C_6(m,\|ZV\|_{L^\infty})$ such that
		\begin{equation}\label{3b}
			h(r_2) \leq \left(\frac{r_3}{2r_2}\right)^{C_6\|V\|_{L^\infty}^2}h(r_3)^{\frac{\beta_0}{\alpha_0+\beta_0}}h(r_1)^{\frac{\alpha_0}{\alpha_0+\beta_0}},
		\end{equation}
		where
		\[
		\alpha_0 = \log\left( \frac{r_3}{2r_2} \right), \ \ \beta_0 = C_5^{2\|V\|_{L^\infty}}\log\left(\frac{2r_2}{r_1}\right).
		\]
	\end{theorem}
	\begin{proof}
		The almost monotonicity of $N(r)$ holds directly
		\begin{equation*}
			\begin{aligned}
				e^{C_3K_1r^2}\left(N(r)+C_4K_1^2(\alpha+1)+\frac{1}{4C_3}\frac{K_2}{K_1}\right)\leq e^{C_3K_1s^2}\left(N(s)+C_4K_1^2(\alpha+1)+\frac{1}{4C_3}\frac{K_2}{K_1}\right),
			\end{aligned}
		\end{equation*}
		for any $0<r<s<1$, which implies
		\begin{equation}\label{Nrs}
			\begin{aligned}
				N(r) &\leq \frac{e^{C_3 K_1s^2}}{e^{C_3 K_1 r^2}}\left(N(s) + C_4 K_1^2 (\alpha+1)+\frac{1}{4C_3}\frac{K_2}{K_1} \right)\\[2mm]
				&\leq C_5^{K_1} \left(N(s) + C_4 K_1^2 (\alpha+1)+\frac{1}{4C_3}\frac{K_2}{K_1} \right),\ \ \ \ \ \ \ \ 0<r<s<1,
			\end{aligned}
		\end{equation}
		where $C_5 = e^{C_3}$.
		By the first variation formula of $H(r)$ in Lemma \ref{H1}, one has
		\begin{equation}\label{dr}
			\frac{d}{dr} \log\left(\frac{H(r)}{r^{2 \alpha+ Q }}\right) = \frac{1}{(\alpha +1)r} N(r),\ \ \ \ 0<r<1.
		\end{equation}
		Integrating (\ref{dr}) between $r_1$ and $2r_2$, using (\ref{Nrs}), we obtain
		\begin{equation*}
			\begin{aligned}
				&\quad \log \frac{H(2r_2)}{H(r_1)} - (2\alpha + Q ) \log \left(\frac{2r_2}{r_1}\right) = \frac{1}{(\alpha +1)} \int_{r_1}^{2r_2} \frac{1}{r} N(r) dr \\[2mm]
				&\leq \frac{C_5^{K_1}}{\alpha + 1} \left(N(2r_2) + C_4 K_1^2 (\alpha+1)+\frac{1}{4C_3}\frac{K_2}{K_1}\right) \int_{r_1}^{2r_2} \frac{1}{r} dr,
			\end{aligned}
		\end{equation*}
		which implies
		\begin{equation}\label{12}
			\frac{\log \frac{H(2r_2)}{H(r_1)}}{\log \left(\frac{2r_2}{r_1}\right)} - (2\alpha + Q )  \le   \frac{C_5^{K_1}}{\alpha + 1} \left(N(2r_2) + C_4 K_1^2 (\alpha+1)+\frac{1}{4C_3}\frac{K_2}{K_1}\right).
		\end{equation}
		Next we use (\ref{Nrs}) again to integrate (\ref{dr}) between $2r_2$ and $r_3$, similarly we find
		\begin{equation}\label{21}
			\begin{aligned}
				\frac{C_5^{K_1}}{\alpha + 1} \left(N(2r_2) - C_5^{K_1}C_4 K_1^2 (\alpha+1) - \frac{C_5^{K_1}}{4C_3}\frac{K_2}{K_1}\right) \le C_5^{2K_1} \left[\frac{\log \frac{H(r_3)}{H(2r_2)}}{\log\left(\frac{r_3}{2r_2}\right)}  - (2\alpha + Q )\right].
			\end{aligned}
		\end{equation}
		Thus by (\ref{12}) and (\ref{21}), we conclude that
		\[
		\frac{\log \frac{H(2r_2)}{H(r_1)}}{C_5^{2K_1} \log \left(\frac{2r_2}{r_1}\right)}   \le  \frac{\log \frac{H(r_3)}{H(2r_2)}}{
			\log\left(\frac{r_3}{2r_2}\right)} + 2 C_4K_1^2 + \frac{1}{\alpha+1} \frac{K_2}{2C_3K_1} - \left(1 - \frac{1}{C_5^{2K_1}}\right)(2\alpha + Q ),
		\]
		since $\frac{C_5^{K_1}+1}{C_5^{K_1}} \leq 2$.
		Now setting
		\[
		\alpha_0 = 
		\log\left(\frac{r_3}{2r_2}\right),\ \ \ \beta_0 = C_5^{2K_1} \log \left(\frac{2r_2}{r_1}\right),
		\]
		noting that $C_5 \geq 1$, we find
		\begin{equation}\label{alpha0}
			\alpha_0 \log \frac{H(2r_2)}{H(r_1)} \le \beta_0  \log \frac{H(r_3)}{H(2r_2)} +\left(  C_6K_1^2 + \frac{1}{\alpha+1} \frac{K_2}{2C_3K_1} \right) \alpha_0 \beta_0 ,
		\end{equation}
		where $C_6 = 2C_4$.
		Dividing both sides of the last inequality by $\alpha_0 + \beta_0$, we get
		\[
		\log \left(\frac{H(2r_2)}{H(r_1)}\right)^\frac{\alpha_0}{\alpha_0 + \beta_0} \le   \log \left(\frac{H(r_3)}{H(2r_2)}\right)^\frac{\beta_0}{\alpha_0 + \beta_0} + \left( C_6K_1^2 + \frac{1}{\alpha+1} \frac{ K_2}{2C_3K_1}\right) \frac{\alpha_0 \beta_0}{\alpha_0 + \beta_0},
		\]
		which gives
		\begin{equation}\label{l}
			\log H(2r_2) \leq \log\left[H(r_3)^\frac{\beta_0}{\alpha_0 + \beta_0} H(r_1)^\frac{\alpha_0}{\alpha_0 + \beta_0}\right] +\left(  C_6K_1^2 + \frac{1}{\alpha+1} \frac{K_2}{(8M_1 + 2)K_1} \right)  \alpha_0,
		\end{equation}
		since $\frac{\alpha_0 \beta_0}{\alpha_0 + \beta_0}\leq\alpha_0$.
		Exponentiating both sides of (\ref{l}), and let $\alpha$ be sufficient large, we conclude
		\begin{equation*}
			H(2r_2) \leq \left(\frac{r_3}{2r_2}\right)^{\widetilde{C_6}K_1^2} H(r_3)^\frac{\beta_0}{\alpha_0 + \beta_0} H(r_1)^\frac{\alpha_0}{\alpha_0 + \beta_0},
		\end{equation*}
		and $\widetilde{C_6}$ depends on $K_2$.
		Applying (\ref{hge}) and (\ref{hle}), it finally holds (\ref{3b}).
	\end{proof}
	
	\section{The maximal vanishing order of solutions}\label{maxi}
	In this section, we demonstrate Theorem \ref{thm}, which is the main result in this paper. The theorem tells that if $u$ is a nontrivial solution of (\ref{equ1}), then the maximal vanishing order is not larger than a finite number. Before our proving of Theorem \ref{thm}, we first provide some  a priori interior estimates for solutions of the equations (\ref{equ}), which can help us to find the relationship between $h(r)$ and $\left\|u\right\|_{L^\infty}$.
	
	First we prove the following Caccioppoli-type inequality, which has been obtained for the standard bi-Laplacian equation in \cite{L}.
	
	\begin{lemma}\label{Cacci}
		Let $u$ and $w$ satisfy the equations \eqref{equ}, then there exists a positive universal $C$ such that
		\begin{equation*}
			\left\|w\right\|_{L^2(B_r)}^2 \leq C\left(1 + \|V\|_{L^\infty} \right) r^{-4} \left\|u\right\|_{L^2(B_{2r})}^2,
		\end{equation*}
		for any $0<r<\frac{1}{2}$.
	\end{lemma}
	\begin{proof}
		Supposing $u$ and $w$ are solutions of the equation (\ref{equ}), then we can get
		\begin{equation*}
			\int_{B_{2r}}\Delta_{X}u\Delta_{X}\varphi = \int_{B_{2r}}Vu\varphi,
		\end{equation*}
		for any $\varphi \in L_0^2(B_{2r})$ satisfies $X\varphi \in L_0^2(B_{2r})$ and $\Delta_{X}\varphi \in L^2(B_{2r})$.
		Let $\eta \in C_0^\infty(\mathbb{R})$ be a cut-off function which satisfies
		\begin{align}
			&0\leq\eta\leq1,\notag\\[2mm]
			\eta(\rho)&\equiv1\ \ \ \ \mbox{for}\ \ \rho\leq r,\notag
		\end{align}
		and
		\begin{align}\label{eta}
			\eta(\rho) = 0 \quad \mbox{for}\ \ \rho\geq2r; \ \ \left|X\eta(\rho)\right|\leq\frac{C}{r}, \ \ \left|X^2\eta(\rho)\right|\leq\frac{C}{r^2} \ \ \ \ \mbox{for}\ \ r < \rho\leq 2r.
		\end{align}
		We can assume $C>1$.
		Taking test function $\varphi = u\eta^4$, one has
		\begin{equation*}
			\int_{B_{2r}}\Delta_{X}u\Delta_{X}(u\eta^4) = \int_{B_{2r}}Vu^2\eta^4,
		\end{equation*}
		%	Calculation: 	$$\Delta_X (u \eta^4) = \mbox{div}_X (X (u\eta^4)) = \mbox{div}_X (\eta^4 X u + 4\eta^3X \eta u) = 4\eta^3 X \eta \cdot X u + \eta^4 \Delta_X u + 12\eta^2 |X\eta|^2 u + 4\eta^3 u \Delta_X \eta + 4\eta^3 X\eta \cdot Xu$$
		which implies
		\begin{align*}
			\int_{B_{2r}}\left(\Delta_{X}u\right)^2\eta^4 &=  \int_{B_{2r}}Vu^2\eta^4 -4\int_{B_{2r}}\eta^3u\Delta_{X}u\Delta_{X}\eta \notag\\[2mm]
			& \quad - 12\int_{B_{2r}}\eta^2u\Delta_{X}u\left|X\eta\right|^2 -8\int_{B_{2r}}\eta^3\Delta_{X}u\langle Xu,X\eta\rangle.
		\end{align*}
		By Young's inequality,
		\begin{align}\label{Du2}
			\int_{B_{2r}}\left(\Delta_{X}u\right)^2\eta^4 &\leq \int_{B_{2r}}Vu^2\eta^4 + \epsilon\int_{B_{2r}}\left(\Delta_{X}u\right)^2\eta^4 \notag\\[2mm]
			& \quad + C(\epsilon)\left\{ \int_{B_{2r}}\eta^2u^2\left(\Delta_{X}\eta\right)^2 + \int_{B_{2r}}u^2\left|X\eta\right|^4 + \int_{B_{2r}}\eta^2\left|Xu\right|^2\left|X\eta\right|^2\right\},
		\end{align}
		and we let $\epsilon=\frac{1}{4}$.
		Next, using Young's inequality again, we estimate the last term in (\ref{Du2}).
		\begin{align*}
			\int_{B_{2r}}\eta^2\left|Xu\right|^2\left|X\eta\right|^2 &= -\int_{B_{2r}}\mbox{div}_X\left(Xu\left|X\eta\right|^2\eta^2\right)u\notag\\[2mm]
			& = -\int_{B_{2r}}\Delta_{X}u\left|X\eta\right|^2\eta^2u - 2 \sum_{j,k=1}^{m+n}\int_{B_{2r}}X_kuX_k(X_j\eta)X_j\eta\eta^2u\notag\\[2mm]
			& \quad - 2\int_{B_{2r}}\eta u\left|X\eta\right|^2 \langle Xu,X\eta\rangle\notag\\[2mm]
			& \leq \frac{1}{8}\int_{B_{2r}}\left|\Delta_Xu\right|^2\eta^4 + \frac{1}{2}\int_{B_{2r}}\left|Xu\right|^2\left|X\eta\right|^2\eta^2\notag\\[2mm]
			& \quad + C\left\{\int_{B_{2r}}u^2\left|X\eta\right|^4 + \int_{B_{2r}}\left|X^2\eta\right|^2\eta^2u^2\right\},
		\end{align*}
		which gives
		\begin{equation}\label{last}
			\int_{B_{2r}}\eta^2\left|Xu\right|^2\left|X\eta\right|^2 \leq \frac{1}{4}\int_{B_{2r}}\left|\Delta_{X}u\right|^2\eta^4 + C\left(\int_{B_{2r}}u^2\left|X\eta\right|^4 + \int_{B_{2r}}u^2\left|X^2\eta\right|^2\eta^2\right).
		\end{equation}
		Thus, combining (\ref{Du2}) and (\ref{last}), applying (\ref{eta}), we conclude
		\begin{equation*}
			\int_{B_{2r}}\left|\Delta_{X}u\right|^2\eta^4 \leq \int_{B_{2r}}Vu^2\eta^4 + Cr^{-4}\int_{B_{2r}}u^2.
		\end{equation*}
		Thanks to the equation (\ref{equ}), we obtain
		\begin{align*}
			\left\|w\right\|_{L^2(B_r)}^2 &\leq \int_{B_{2r}}w^2\eta^4  = \int_{B_{2r}}\left( \Delta_{X} u\right) ^2 \eta^4 \leq \int_{B_{2r}}Vu^2\eta^4 + Cr^{-4}\int_{B_{2r}}u^2.
		\end{align*}
		which implies the result of this lemma.
	\end{proof}
	
	Then we want to apply the Moser  iterative approach to the Baouendi-Grushin equations which are degenerate elliptic. In fact, we can see such process is  available to the Baouendi-Grushin operator, and yields the following result.
	\begin{lemma}\label{moser}
		Let $s>\frac{Q}{2}$, and $w$ bounded satisfies the equation
		\begin{equation}\label{moserr}
			\Delta_{X}w  = f,\ \ \mbox{in}\ B_1(0),
		\end{equation}
		where $f \in L^s(B_1)$,
		then $w\in L^\infty(B_{1/2})$. Moreover, there exists a positive number $\bar{C} = \bar{C}(Q,s)$ such that
		\begin{equation*}
			\left\|w\right\|_{L^\infty(B_{1/2})} \leq \bar{C }\big(\left\|w\right\|_{L^2(B_1)} + \left\|f\right\|_{L^s(B_1)}\big).
		\end{equation*}
		\begin{proof}
			For some $k>0$ and $t>0$, set $\bar{w} =  w^+ + k$ and
			\[
			\bar{w}_t =
			\begin{cases}
				\bar{w} & \text{if } w<t,\\
				k+t & \text{if } w \geq t.
			\end{cases}
			\]
			Then we have 
			\begin{align}\label{Xwm}
				X\bar{w}_t=0\quad\text{in}\ \lbrace t<0\rbrace\cup\lbrace w>t\rbrace.
			\end{align}
			Let
			\[
			\varphi = \eta^2\left(\bar{w}_t^ \zeta \bar{w}-k^{ \zeta +1}\right) \in H_0^1(B_1)
			\]
			for some $ \zeta \geq0$ and some nonnegative function $\eta\in C_0^1(B_1)$. Direct calculation implies
			\begin{align*}
				X\varphi &= \zeta \eta^2\bar{w}_t^{ \zeta -1}X\bar{w}_t\bar{w} + \eta^2\bar{w}_t^ \zeta X\bar{w} + 2\eta X\eta\left(\bar{w}_t^ \zeta \bar{w} - k^{ \zeta +1}\right)\\[2mm]
				&=\eta^2\bar{w}_t^ \zeta \left( \zeta X\bar{w}_t + X\bar{w}\right) + 2\eta X\eta\left(\bar{w}_t^ \zeta \bar{w} - k^{ \zeta +1}\right).
			\end{align*}
			Noting that $\varphi = 0$ and $X\varphi = 0$ in $\lbrace w\leq0\rbrace$, applying \eqref{Xwm} and H\"older's inequality, we have
			\begin{align}\label{Xw}
				\int \langle Xw,X\varphi\rangle &= \int \langle X\bar{w} , \zeta X\bar{w}_t + X\bar{w}\rangle \eta^2 \bar{w}_t^\zeta + 2\int \langle X\bar{w},X\eta\rangle\left(\bar{w}_t^\zeta\bar{w} - k^{\zeta+1}\right)\eta\notag\\[2mm]
				&\geq \beta \int \eta^2\bar{w}_t^\zeta |X\bar{w}_t|^2 + \int \eta^2\bar{w}_t^\zeta |X\bar{w}|^2 - 2\int |X\bar{w}| |X\eta|\bar{w}_t^\zeta \bar{w}\eta\notag\\[2mm]
				&\geq \beta \int \eta^2\bar{w}_t^\zeta |X\bar{w}_t|^2 + \frac{1}{2}\int \eta^2\bar{w}_t^\zeta |X\bar{w}|^2 - 2\int |X\eta|^2\bar{w}_t^\zeta \bar{w}^2.
			\end{align}
			Via \eqref{moserr}, we get $\int \langle Xw,X\varphi\rangle = \int (cw-f)\varphi$, hence \eqref{Xw} implies by noting $\bar{w}\geq k$
			\begin{align}\label{b}
				&\beta \int \eta^2\bar{w}_t^\zeta |X\bar{w}_t|^2 + \int \eta^2\bar{w}_t^\zeta |X\bar{w}|^2\notag\\[2mm]
				&\leq 4 \left\{\int|X\eta|^2\bar{w}_t^\zeta\bar{w}^2 +  |f|\eta^2\bar{w}_t^\zeta\bar{w} \right\}\notag\\[2mm]
				&\leq 4 \left\{\int|X\eta|^2\bar{w}_t^\zeta\bar{w}^2 + \int c_0\eta^2 \bar{w}_t^\zeta\bar{w}^2 \right\},
			\end{align}
			where $c_0$ is denoted by
			\begin{align*}
				c_0 =  \frac{|f|}{k}.
			\end{align*}
			Choose $k = \|f\|_{L^s}$ if $f$ is not identically zero. Otherwise choose arbitrary $k>0$ and eventually let $k \rightarrow 0^+$. Thus we have
			\begin{align}\label{c0}
				\|c_0\|_{L^s}\leq 1.
			\end{align}
			Set $v = \bar{w}_t^{\frac{\zeta}{2}}\bar{w}$. Note that
			\begin{align*}
				|Xv|^2 \leq (1+\zeta)\left\{\zeta \bar{w}_t^\zeta|X\bar{w}_t|^2 + \bar{w}_t^\zeta|X\bar{w}|^2\right\}.
			\end{align*}
			Therefore by \eqref{b} we have
			\begin{equation*}
				\int |Xv|^2\eta^2 \leq  \left\{(1+\zeta) \int v^2|X\eta|^2 + (1+\zeta)\int c_0v^2\eta^2\right\},
			\end{equation*}
			\begin{equation}\label{333}
				\int |X(v\eta)|^2 \leq 4 \left\{(1+\zeta) \int v^2|X\eta|^2 + (1+\zeta)\int c_0v^2\eta^2\right\}.
			\end{equation}
			H\"older's inequality and \eqref{c0} imply that
			\begin{align}\label{111}
				\int c_0v^2\eta^2 &\leq \left(\int c_0^s\right)^\frac{1}{s} \left(\int (\eta v)^\frac{2s}{s-1}\right)^{1-\frac{1}{s}}\notag\\[2mm]
				&\leq \left(\int (\eta v)^\frac{2s}{s-1}\right)^{1-\frac{1}{s}}.
			\end{align}
			Applying interpolation inequality, Young's inequality and Sobolev's inequality for Grushin-type operators(see\cite{FL}), we have
			%   $$\theta = \frac{1+Q}{2s}$$
			\begin{align}\label{222}
				\|\eta v\|_{L^\frac{2s}{s-1}} &\leq \epsilon\|\eta v\|_{L^{2^*}} + C\epsilon^{-\frac{Q}{2s-Q}}\|\eta v\|_{L^2}\notag\\[2mm]
				&\leq \epsilon\|X(\eta v)\|_{L^2} + C\epsilon^{-\frac{Q}{2s-Q}}\|\eta v\|_{L^2}
			\end{align}
			for any small $\epsilon > 0$, where $2^* = \frac{2Q}{Q-2} > \frac{2s}{s-1} > 2$ since $s>\frac{Q}{2}$.
			Therefore plugging \eqref{111}, \eqref{222} into \eqref{333}, we obtain
			\begin{align*}
				\int |X(v\eta)|^2 \leq C \left\{(1+\zeta)\int v^2 |X\eta|^2 + (1+\zeta)^\frac{2s+Q}{2s-Q} \int v^2\eta^2\right\},
			\end{align*}
			which implies
			\begin{align*}
				\int |X(v\eta)|^2 \leq C  (1+\zeta)^\alpha \int (|X\eta|^2 + \eta^2)v^2,
			\end{align*}
			where $\alpha$ is a positive number depending only on $Q$ and $s$. Then Sobolev embedding implies
			\begin{align*}
				\left(\int |\eta v|^{2\chi}\right)^\frac{1}{\chi} \leq C  (1+\zeta)^\alpha \int (|X\eta|^2 + \eta^2)v^2,
			\end{align*}
			where $\chi = \frac{Q}{Q-2} > 1$. For any $0<r<R\leq1$ set $\eta\in C_0^1(B_R)$ with
			\begin{align*}
				0\leq\eta\leq1, \ \ \eta \equiv 1\ \ \mbox{in}\ B_r \ \ \ \mbox{and} \ \ \ |X\eta|\leq\frac{2}{R-r}.
			\end{align*}
			Then we obtain
			\begin{align*}
				\left(\int_{B_r} v^{2\chi}\right)^\frac{1}{\chi} \leq C  \frac{(1+\zeta)^\alpha}{(R-r)^2} \int_{B_R}v^2.
			\end{align*}
			Recalling the definition of $v$, we have
			\begin{align*}
				\left(\int_{B_r} \bar{w}^{2\chi} \bar{w}_t^{\zeta\chi} \right)^\frac{1}{\chi} \leq C  \frac{(1+\zeta)^\alpha}{(R-r)^2} \int_{B_R}\bar{w}^{2} \bar{w}_t^{\zeta}.
			\end{align*}
			Set $\ga = \zeta+2 \geq 2$. Then noting that $\bar{w}_t\leq\bar{w}$, we have
			\begin{align*}
				\left(\int_{B_r} \bar{w}_t^{\ga\chi} \right)^\frac{1}{\chi} \leq C  \frac{(\ga-1)^\alpha}{(R-r)^2} \int_{B_R}\bar{w}^{\ga}
			\end{align*}
			provided the integral in the r.h.s. is bounded. By letting $t\rightarrow+\infty$ we conclude that
			\begin{align*}
				\|\bar{w}\|_{L^{\ga\chi}(B_r)} \leq \left(C \frac{(\ga-1)^\alpha}{(R-r)^2}\right)^\frac{1}{\ga} \|\bar{w}\|_{L^\ga(B_R)}
			\end{align*}
			provided $\|\bar{w}\|_{L^\ga(B_R)} < +\infty$, where $C = C(Q,s)$ is a positive constant independent of $\ga$. We iterate the above estimate, beginning with $\ga=2$, as $2, 2\chi, 2\chi^2, \cdots$. Now set for $i = 0, 1, 2,\cdots$, 
			\begin{align*}
				\ga_i = 2\chi^i \quad \mbox{and}\quad r_i = \frac{1}{2} + \frac{1}{2^{i+1}}.
			\end{align*}
			By $\ga_i = \chi\ga_{i-1}$ and $r_{i-1}-r_i = \frac{1}{2^{i+1}}$, we have for $i = 1, 2,\cdots$,
			\begin{align*}
				\|\bar{w}\|_{L^{\ga_i}(B_{r_i})} \leq C^\frac{i}{\chi^i}\|\bar{w}\|_{L^{\ga_{i-1}}(B_{\ga_{i-1}})}
			\end{align*}
			provided $\|\bar{w}\|_{L^{\ga_i-1}(B_{\ga_i-1})} < +\infty$. Hence by iteration we obtain
			\begin{align*}
				\|\bar{w}\|_{L^{\ga_i}(B_{r_i})} \leq C^{\Sigma\frac{i}{\chi^i}}\|\bar{w}\|_{L^2(B_1)},
			\end{align*}
			in particular
			\begin{align}\label{it}
				\left(\int_{B_{1/2}} \bar{w}^{2\chi^i}\right)^\frac{1}{2\chi^i} \leq \bar{C}\left(\int_{B_1}\bar{w}^2\right)^\frac{1}{2},
			\end{align}
			where $\bar{C} = \bar{C}(Q,s)$, noting the r.h.s. of \eqref{it} is finite since $w$ is bounded.
			Letting $i\rightarrow+\infty$ we get
			\begin{align*}
				\sup_{B_{1/2}} \bar{w} \leq \bar{C}\|\bar{w}\|_{L^2(B_1)},
			\end{align*}
			hence                                   
			\begin{align*}
				\sup_{B_{1/2}} w^+ \leq \bar{C}{\|w^+\|_{L^2(B_1)} + k}.
			\end{align*}
			Recalling the definition of $k$, the proof is completed.
		\end{proof}
	\end{lemma}
	\begin{remark}
		Although the Baouendi-Grushin operator $\Delta_{X}$ is not uniformly elliptic, it still holds that
		\begin{equation}\label{a}
			\int a^{ij} X_i u X_j u \geq \lambda \int \left|Xu\right|^2,
		\end{equation}
		when the coefficient matrix $\left(a^{ij}\right)$ is positive, where $\lambda$ is a positive number. Thanks to \eqref{a} and the Sobolev imbedding, the proof of Lemma \rm{\ref{moser}} is similar to which of Theorem \rm{4.1} in \cite{LH}.
	\end{remark}
	At last we are ready to prove Theorem \ref{thm}.
	\begin{proof}[Proof of Theorem \ref{thm}]
		For $R_1<1$, we assume $C_0$ by
		\begin{equation*}
			C_0 = \bigg(\left\|u\right\|_{L^\infty(B_{R_1})}^2 + \left\|w\right\|_{L^\infty(B_{R_1})}^2\bigg) \int_{B_{R_1}}\psi > 0,
		\end{equation*}
		then clearly we have $h(R_1)\leq C_0$.
		In Theorem \ref{tb}, taking $r_2=\frac{R_1}{3}$, and $r_3=R_1$, we get
		\begin{equation*}
			h\left(\frac{R_1}{3}\right)^{1+\frac{\beta_0}{\alpha_0}} \leq \left(\frac{3}{2}\right)^{\widetilde{C_6}K_1^2\left(1+\frac{\beta_0}{\alpha_0}\right)}C_0^{\frac{\beta_0}{\alpha_0}}h(r),\ \ \ \ 0<r<\frac{R_1}{3}.
		\end{equation*}
		Set $q:=\frac{\beta_0}{\alpha_0}= \frac{C_5^{2K_1}\log\left(\frac{2R_1}{3r}\right)}{\log\left(\frac{3}{2}\right)} = -\log\left( \frac{r}{R_1}\right)^\ga - C_5^2$,
		where $\ga:=\frac{C_5^{2K_1}}{\log\left(\frac{3}{2}\right)}$.
		Recalling that $C_5\geq1$, we obtain for $0<r<\frac{R_1}{3}$,
		\begin{align*}
			h(r) &\geq C_0\left(\frac{h(\frac{R_1}{3})}{C_0}\right)^{1+q}\left(\frac{3}{2}\right)^{-\widetilde{C_6}K_1^2(1+q)}\notag\\[2mm]
			&\geq C_0M_0^{1+q}\left(\frac{r}{R_1}\right)^{\widetilde{C_6} C_7^{K_1} K_1^2},
		\end{align*}
		where $M_0:=\frac{h(\frac{R_1}{3})}{C_0}$, $C_7 = C_5^2.$
		
		Next, we distinguish the following two conditions.
		If $M_0\geq1$,
		\begin{equation*}
			h(r)\geq C_0\left(\frac{r}{R_1}\right)^{\widetilde{C_6} C_7^{K_1} K_1^2}.
		\end{equation*}
		If instead $0<M_0\leq1$, keeping in mind that $C_5\geq1$ and $K_1\geq1$, with $C_8:=\max\left\{\widetilde{C_6},\frac{\log\left(\frac{1}{M_0}\right)}{\log\left(\frac{3}{2}\right)}\right\}$, we obtain
		\begin{align}\label{3}
			h(r) &\geq C_0\left(\frac{r}{R_1}\right)^{\widetilde{C_6} C_7^{K_1}K_1^2+\ga\log\left(\frac{1}{M_0}\right)}\notag\\[2mm]
			&\geq C_0\left(\frac{r}{R_1}\right)^{C_8 C_7^{K_1} (1+K_1^2)}\notag\\[2mm]
			&\geq C_0\left(\frac{r}{R_1}\right)^{2C_8 C_7^{K_1} K_1^2}.
		\end{align}
		By Lemma \ref{Cacci}, we get
		\begin{equation}\label{x}
			\left\|w\right\|_{L^2(B_r)} \leq C\sqrt{1 + K_1}r^{-2}\left\|u\right\|_{L^2(B_{2r})}.
		\end{equation}
		If $u$ and $w$ satisfy the equations (\ref{equ}), by virtue of Lemma \ref{moser}, choosing $s = \frac{Q+1}{2}$, we get
		\begin{equation*}
			\left\|w\right\|_{L^\infty(B_{1/2})} \leq C(Q) \big(\left\|w\right\|_{L^2(B_1)} + K_1 \left\|u\right\|_{L^{\frac{Q+1}{2}}(B_1)}\big),
		\end{equation*}
		which implies by scaling
		\begin{equation}\label{y}
			\left\|w\right\|_{L^\infty(B_{r/4})} \leq C(Q) r^{-\frac{Q}{2}} \left\|w\right\|_{L^2(B_{r/2})} + C(Q)K_1 r^{-\frac{2Q}{Q+1}} \left\|u\right\|_{L^{\frac{Q+1}{2}}(B_{r/2})}.
		\end{equation}
		Thus by applying \eqref{x} and \eqref{y}, one has
		\begin{align*}
			\left\|w\right\|_{L^\infty(B_{r/4})} &\leq C(Q) \sqrt{K_1} r^{-\frac{Q+4}{2}}\left\|u\right\|_{L^2(B_r)} + C(Q)K_1 r^{-\frac{2Q}{Q+1}} \left\|u\right\|_{L^{\frac{Q+1}{2}}(B_{r/2})}\notag\\[2mm]
			& \leq C(Q) \sqrt{K_1} r^{-2}\left\|u\right\|_{L^\infty (B_r)} + C(Q)K_1  \left\|u\right\|_{L^\infty (B_{r/2})}\notag\\[2mm]
			& \leq C(Q)K_1 r^{-2}\left\|u\right\|_{L^\infty(B_r)},
		\end{align*}
		which implies
		\begin{align}\label{2}
			h(\frac{r}{4}) &= \int_{B_{r/4}}(u^2+w^2)\psi \notag\\[2mm]
			&\leq \bigg(\left\|u\right\|^2_{L^\infty(B_{r/4})} + \left\|w\right\|^2_{L^\infty(B_{r/4})}\bigg) \int_{B_{r/4}}\psi \notag\\[2mm]
			& \leq \widetilde{C}r^{-4} \left\|u\right\|_{L^\infty(B_r)}^2,
		\end{align}
		where $\widetilde{C} = C(Q)K_1 \int_{B_{r/4}}\psi$ depending on $Q$, $K_1$ and $\beta$.
		Substituting (\ref{2}) into (\ref{3}) and letting $R_1 \rightarrow 1$ , we finally get for $0<r<1$,
		\begin{equation*}
			\left\|u\right\|_{L^\infty(B_{r})} \geq 4^{-2C_8 C_7^{K_1} K_1^2}\widetilde{C}^{-1}C_0 r^{2C_8 C_7^{K_1} K_1^2+4}.
		\end{equation*}
		Recalling $K_1\geq1$ and $C_7^{K_1}\geq 1$, it holds that $2C_8 C_7^{K_1} K_1^2+4 \leq C_7^{K_1}C_9(K_1^2+1)\leq 2C_9C_7^{K_1} K_1^2$, where $C_9 = \max\left\{2C_8, 4\right\}$, hence
		\begin{equation*}
			\left\|u\right\|_{L^\infty(B_{r})} \geq C_{10}r^{C_{11} C_7^{K_1} K_1^2},
		\end{equation*}
		where $C_{10} := 4^{-2C_8 C_7^{K_1} K_1^2}\widetilde{C}^{-1}C_0$ and $C_{11} := 2C_9$, the proof completed.
	\end{proof}
	
	\subsection*{Acknowledgments}
	This work was supported by the National Natural Science Foundation of China (No. 12161044) and the Jiangxi Provincial Natural Science Foundation (Nos. 20242BCE50042, 20224ACB218001 and 20224BCD41001).

	\end{document}